\documentclass[twoside,11pt,reqno]{amsart}
\usepackage{xypic}

\usepackage{tikz}
\usepackage{pgfplots}
\pgfplotsset{compat=1.14}
\usepackage{mathrsfs}
\usetikzlibrary{arrows}

\usepackage[bookmarks,
bookmarksnumbered,%
 colorlinks=true,%
 linkcolor=red,%
 citecolor=blue,%
 filecolor=blue,%
 urlcolor=blue,%
]
{hyperref}

\usepackage{amssymb, amsmath, latexsym}
\usepackage{epigraph}

\makeatletter
\@namedef{subjclassname@2020}{\textup{2020} Mathematics Subject Classification}
\makeatother

\DeclareMathAlphabet{\curly}{OT1}{rsfs}{n}{it}
 
\DeclareMathOperator{\rank}{rank}

\DeclareMathOperator{\im}{Im}

\DeclareMathOperator{\Fix}{Fix}
\DeclareMathOperator{\tr}{tr}

\DeclareMathOperator{\inc}{in}
\DeclareMathOperator{\pr}{pr}
\DeclareMathOperator{\coker}{Coker}
\DeclareMathOperator{\Sm}{Sm}

\DeclareMathOperator{\Bl}{Bl}
\DeclareMathOperator{\Tors}{Tors}
\DeclareMathOperator{\interior}{Int}

\DeclareMathOperator{\Conj}{c}

\def\dim{\mbox{dim}}
\def\ra{\rightarrow}
\def\cal{\mathcal}

\def\CC{\mathbb{C}}
\def\PP{\mathbb{P}}
\def\QQ{\mathbb{Q}}
\def\ZZ{\mathbb{Z}}
  
\def\RR{\mathbb{R}}

\def\FF{\mathbb{F}}

\def\s-{\setminus}

\def\HH{\mathbb{H}}

\def\ra{\rightarrow}

\newtheorem{thm}{Theorem}[section]
\newtheorem{prop}[thm]{Proposition}

\newtheorem{defn}[thm]{Definition}

\newtheorem{cor}[thm]{Corollary}
\newtheorem{lem}[thm]{Lemma}

\newtheorem{rmk}[thm]{Remark}

\numberwithin{equation}{section}

\newenvironment{notations}{\mbox{ }\\{\bf  Notations and conventions:}\mbox{ }}

\begin{document}

\definecolor{ffqqqq}{rgb}{1,0,0}

\title[Maximality]
{On the maximality problem for the Hilbert square of real surfaces}

\author[Kharlamov]{Viatcheslav Kharlamov}

\address{
        IRMA UMR 7501, Strasbourg University, 7 rue Ren\'e-Descartes, 67084 Strasbourg Cedex,  FRANCE}

\email{kharlam@math.unistra.fr}

\author[R\u asdeaconu]{Rare\c s R\u asdeaconu}

\address{        
        Department of Mathematics,1326 Stevenson Center, Vanderbilt University, Nashville, TN, 37240, USA}
  \address{
	Institute of Mathematics of the Romanian Academy,  P.O. Box 1-764, Bucharest 014700,  Romania}

\email{rares.rasdeaconu@vanderbilt.edu}



\begin{abstract}
 We explore maximality with respect to the classical Smith bound on the total Betti 
 number of the real locus. For a large class of surfaces, we prove that the Hilbert 
 square of a real surface is maximal if and only if the surface is maximal and has 
 connected real locus. In particular, the Hilbert square of no K3 or abelian surface 
 is maximal. We also exhibit various types of maximal surfaces, including ones with 
 disconnected real locus, whose Hilbert square is maximal.
\end{abstract}

\dedicatory{Dedicated to the memory of Mark Sapir}

\maketitle

\thispagestyle{empty}

\vskip-4mm
\setlength\epigraphwidth{.53\textwidth}
\epigraph{On d\'edaigne volontiers un but qu'on n'a pas r\'eussi \`a atteindre, 
ou qu'on a atteint d\'efinitivement.}{ M.~Proust, A la recherche du temps perdu. \\}

\section{Introduction}

 A result of fundamental importance in understanding the topology of real algebraic varieties is the 
Smith inequality  \cite{Smith} that bounds from above the sum of the $\FF_2$-Betti numbers
of the fixed point set of an involution by the sum of $\FF_2$-Betti numbers of the ambient space 
itself. Applied to the complex conjugation on an  $n$-dimensional algebraic variety $X$ 
defined over the field $\RR$ it states that
\begin{equation}
\label{smith-intro}
\sum_{i=0}^n \beta_i(X(\RR))\leq\sum_{i=0}^{2n} \beta_i(X(\CC)).
\end{equation}

A real algebraic variety realizing equality in the Smith inequality is said to be {\it maximal} or 
an {\it $M$-variety}. The study of $M$-varieties  is one of central themes in the study of the 
topology of real algebraic varieties \cite{dk}.

In dimension one, examples of $M$-curves of arbitrary genus are given by Harnack \cite{harnack}.
In higher dimensions, the question of sharpness is far from being completely understood despite 
the existence of a powerful patchworking method due to O. Viro. The latter one is presumed 
to provide maximal projective hypersurfaces of any dimension and degree, but yet only asymptotic 
(by degree) maximality is achieved \cite{io}. Otherwise, apart from abelian varieties,  flag manifolds 
equipped with the standard real structure, and smooth toric varieties \cite{bfmh}, only sporadic 
examples are known.

An interesting phenomenon was detected in the work of G.~Weichold \cite{wei} and F.~Klein 
\cite{k}. They established a correspondence between the topology of the real locus of a curve 
and that of its Jacobian which, in modern language, shows that the Jacobian of a real algebraic 
curve with non-empty real part is maximal if and only if the curve is maximal (see \cite{gh} 
for a contemporary presentation). Recently, relying on the work of M.~Liu and F.~Schaffhauser 
\cite{liu}, E.~Brugall\'e and F.~Schaffhauser \cite{bs} provided a new insight into Weichold and 
Klein's work by proving that  the moduli spaces of vector bundles of coprime rank and degree 
over a real algebraic curve with non-empty real part are maximal if and only if the base curve itself 
is maximal.

The question of extending the maximality phenomenon to the symmetric products of curves 
was addressed by I.~Biswas and S.~D'Mello in \cite{BdM}, where they obtained partial results. 
A complete answer was found by M.~Franz  \cite{franz} who disclosed a much more general 
result: a symmetric product of a space with an involution having fixed points is maximal if and 
only if the space itself is maximal.

Closely related to the symmetric product is the Hilbert scheme of points. In this paper we begin 
a study of its maximality and observe a rather different behavior. As the following two theorems 
show, already in the case of Hilbert squares of surfaces, while the statement in ``only if" direction 
is preserved, it drastically fails in the opposite direction.
\begin{thm}
\label{converse}
Let $X$ be a real nonsingular projective surface.
If the Hilbert square $X^{[2]}$ is maximal, then $X$ is maximal.
\end{thm}

\begin{thm}
\label{main}
Let $X$ be a maximal real nonsingular projective surface with $H_1(X(\CC),\FF_2)=0.$ Then its 
Hilbert square $X^{[2]}$ is maximal if and only if the real locus $X(\RR)$ of $X$ is connected. 
\end{thm}

As a consequence of Theorem \ref{main}, we get the following results:

\begin{cor}
\label{main-corollary}  
Let $X$ be a real nonsingular projective surface satisfying $H_1(X(\CC),\FF_2)$ $=0$ and 
$h^{2,0}(X(\CC)) >0$, then $X^{[2]}$ is not maximal. In particular, this is the case if $X$ is 
a K3-surface, or a surface birational to it.
\end{cor}

\begin{cor}
\label{rational} 
If $X$ is a maximal real nonsingular rational surface with $X(\RR)$ connected, then 
$X^{[2]}$ is maximal. In particular, this is the case for the projective plane and ruled surfaces
$\PP_{\PP^1}(\cal E),$ where $\cal E$ is a rank 2 vector bundle equipped 
with a real structure that lifts the standard real structure on $\PP^1.$
\end{cor}

When $X(\CC)$ has a positive first Betti number, we have only a partial answer.

\begin{thm}
\label{b1>0}
Let $X$ be a maximal real nonsingular projective surface with  $\Tors_2 H_1(X(\CC), \ZZ)=0$.
Then:
\begin{itemize}
\item[ 1)] If $X(\RR)$ is connected, then the Hilbert square $X^{[2]}$ is maximal.
\item[ 2)] 
If $\beta_0(X(\RR))>1+ \beta_1(X(\CC)),$ then  the Hilbert square $X^{[2]}$ is not maximal.
\end{itemize}
\end{thm}

Examples of maximal surfaces $X$ with torsion free homology and $1<\beta_0(X(\RR))<\beta_1(X)$ 
are given by the ruled surfaces $X=\PP_C(E),$ where $E$ is a rank $2$ complex vector bundle over 
a maximal curve $C$ of positive genus. For such surfaces, the real locus $X(\RR)$ is disconnected, 
and we show that their Hilbert square is maximal, in contrast with situations described in 
Corollary \ref{main-corollary}. On the other hand, if $X$ is a maximal torus of complex dimension 2, 
we notice that $X^{[2]}$ is not maximal while $\beta_0(X(\RR))=\beta_1(X(\CC)).$

For surfaces $X$ with $2$-torsion in the homology with integer coefficients, one can still 
compute the Betti numbers of the Hilbert square $X^{[2]}$  \cite{square}. For real maximal 
Enriques surfaces, we notice that $X^{[2]}$ is never maximal. We should point out that the 
real locus of maximal Enriques surfaces is always disconnected (see \cite{enriques} for a full 
classification of topological types).

In all of the enumerated results, as well as throughout the whole paper, when we speak about the 
maximality of a Hilbert scheme of points $X^{[n]},$ we mean the maximality of $X^{[n]}$ with 
respect to the real structure which is {\it canonically inherited} from a real structure on $X$. Thus, 
non maximality with respect to this kind of real structures does not exclude that $X^{[n]}$ may admit 
another real structure with respect to which it becomes maximal. In particular, at the current stage 
we do not know if there exist (real or complex) $K3,$ or Enriques, surfaces $X$ for which $X^{[n]}$ 
admits however a maximal real structure.

\begin{rmk}
{\rm 
All of the above definitions and the results, including their proofs, literally extend from real algebraic 
setting to compact complex analytic manifolds equipped with an anti-holomorphic involution.
}
\end{rmk}

\begin{rmk} 
{\rm 
Theorem \ref{converse}, under the additional assumption $X(\RR)\ne \emptyset,$ was also obtained by 
L.~Fu in a recent preprint \cite{fu}, where the maximality  of various moduli spaces of 
sheaves/bundles/subschemes on maximal varieties is disclosed.
}
\end{rmk}


\subsection*{Acknowledgements} We are deeply thankful to L.~Fu for useful discussions and 
sharing with us his rich of ideas preprint  \cite{fu} on the subject. The first author acknowledges 
support from the grant ANR-18-CE40-0009 of French Agence Nationale de Recherche. The second 
author acknowledges the support of a Professional Travel Grant from Vanderbilt University in the 
early stages of this project.


\newpage
\notations


\begin{itemize}
\item[ 1)] By a variety equipped with a real structure we mean a pair $(Y,c)$ consisting of a complex 
variety $Y$ and an anti-holomorphic involution $c:Y\ra Y$. 

\item[ 2)]  Let $Y$ be an algebraic variety defined over $\RR,$ and $G$ denote the Galois group 
${\text{Gal}}(\CC/\RR).$ The group $G$ is a cyclic group of order $2$ and  acts on the locus of 
complex points $Y(\CC).$ The non-trivial element of $G$ acts as an anti-holomorphic involution, 
which we will denote by $c,$ and the fixed point set of the action coincides with the set of real points 
of $Y.$ The pair $(Y, c)$ is a variety equipped with a real structure. To mediate between the notations 
traditionally used for varieties equipped with real structures and for algebraic varieties defined over 
$\RR,$ we will use from now on $Y$ to denote the set of complex points, and $Y(\RR)$ the set of 
real points. 

\item[ 3)] Unless  explicitly stated, all the homology and cohomology groups have coefficients in the 
field $\FF_2=\ZZ/2\ZZ$. We use  $\beta_i(\,\cdot\,)$ and $b_i(\,\cdot\,)$  to denote the Betti numbers 
when the coefficients are in $\FF_2$ or in $\QQ,$ respectively. We will use  the notation $\beta_*(\,\cdot\,)$  
and $b_*(\,\cdot\,)$ for the corresponding total Betti numbers.
\end{itemize}


\section{Preliminaries}


\subsection{Smith theory}
\label{Smith.theory}\label{s7}


Most of the results cited in this section are due to P.A. Smith; proofs can be found, e.g., in \cite[Chapter 3]{bredon} 
and \cite[Chapter 1]{dik}.

\smallskip

Throughout the section we
consider a topological space $X$ equipped with a cellular involution $c:X\to X,$ i.e., an involution $c$ that transforms 
cells into cells and acts identically on each invariant cell. Denote by $F=\Fix c,\,\bar X=X/c,$ and  let 
$\inc: F\hookrightarrow X$ and $\pr: X\ra \bar X$ be the natural inclusion and projection, respectively.

Introduce the \emph{Smith chain complexes}
\begin{align*}
\Sm_*(X)&=\ker[(1+c_*)\, :S_*(X)\to S_*(X)],\\
\Sm_*(X,F)&=\ker[(1+c_*)\, :S_*(X,F)\to S_*(X,F)].
\end{align*}
and \emph{Smith homology} $H_r(\Sm_*(X))$ and $H_r(\Sm_*(X,F))$. There is a canonical isomorphism 
$\Sm_*(X,F)=\im[(1+c_*)\, :S_*(X)\to S_*(X)].$ The \emph{Smith sequences} are the long homology and
cohomology exact sequences associated with the short exact sequence of complexes
\begin{equation}
\label{smith-sequence}
0\ra\Sm_*(X)\xrightarrow{\text{inclusion}}S_*(X)
\xrightarrow{1+c_*}\Sm_*(X,F)\ra 0.
\end{equation}

Analyzing this sequence notice, first that we have a canonical canonical splitting 
$\Sm_*(X)=S_*(F)\oplus\im(1+c_*).$ The transfer homomorphism 
$\tr^*:S_*(\bar X,F)\to\Sm_*(X,F)$ is an isomorphism \cite[Chapter 3]{bredon} 
(see also {\it op. cit.} for the cohomology version). In view of the above identifications, 
the long exact sequences associated to (\ref{smith-sequence}) yield:

\begin{thm}
\label{Smith.seq}
There are two natural, in respect to equivariant maps, exact sequences, called  
(homology and cohomology) \emph{Smith sequences} of~$(X,c)$:
$$
\begin{gathered}
\cdots \ra H_{p+1}(\bar X,F)\xrightarrow[]{\Delta} H_p(\bar X,F)\oplus H_p(F)
  \xrightarrow{\tr^*+\inc_*} H_p(X)
 \xrightarrow{\pr_*} H_p(\bar X,F)\ra \rlap{\,},\\
\ra H^p(\bar X,F)\xrightarrow{\pr^*}H^p(X)\xrightarrow{{\tr_*}\oplus{\inc^*}} 
  H^p(\bar X,F)\oplus H^p(F)\xrightarrow{\Delta} H^{p+1}(\bar X,F)\ra\cdots  \rlap{\,.}
\end{gathered}
$$

The homology and cohomology connecting homomorphisms $\Delta$ are given by
$$
x\mapsto x\cap\omega\oplus\partial x\quad\text{and}\quad x\oplus f\mapsto
x\cup\omega+\delta f,
$$
respectively, where $\omega \in H^1(\bar X\setminus F)$ is the characteristic class 
of the double covering $X\setminus F\to\bar X\setminus F$. The images of 
${\tr^*}+\inc_*$ and~$\pr^*$ consist of invariant classes:
$\im\tr^*\subset\ker(1+c_*)$ and $\im\pr^*\subset\ker(1+c^*)$.
\end{thm}

The following immediate consequences of  Theorem \ref{Smith.seq}, 
which we state in the homology setting, 
have an obvious counterpart for cohomology:

\begin{cor}
Let $(X,c)$ be a topological space equipped with a cellular involution. Then:
\label{ineq}
\begin{enumerate}
\item[ 1)] $\dim~H_*(F)+2\sum_{p}\dim\coker({\tr^p}+\inc_p)=\dim~H_*(X)$
(Smith identity);
\item[ 2)] $\dim~H_*(F)\le\dim~H_*(X)$ (Smith inequality);
\end{enumerate}
\end{cor}
Recall that by definition $H^1(\FF_2;H_*(X))=\ker(1+c_*)/\im(1+c_*).$

\begin{defn}
Let $(X,c)$ be a topological space equipped with a cellular involution.
If  $\dim~H_*(F)=\dim~H_*(X)$,
one says that $c$ is an \emph{$M$-involution}, and $X$ is called \emph{maximal}, or an \emph{$M$-variety}. 
\end{defn}

\begin{cor}
\label{maxSmith}
Let $(X,c)$ be a topological space equipped with a cellular involution. Then $X$ is an an $M$-variety 
if and only if for every $k\geq 0,$ the sequence 
$$
0\ra H_{k+1}(\bar X, F)\xrightarrow{\Delta} H_{k}(\bar X, F)\oplus H_k(F)\ra H_k(X)\ra 0
$$
is exact.
\end{cor}


\subsection{Equivariant (co)homology and the Borel spectral sequence}


Let $BG=\RR\PP^{\infty}$ be the classifying space of a cyclic group of order $2$ denoted by $G,$ and 
$EG=S^{\infty}$ its universal cover.  Let $X_G := (X\times EG)/G$ be the Borel construction, where $G$ 
acts freely, as the diagonal action, and denote by $H^*_G(X,\FF_2) := H^*(X_G,\FF_2)$ the equivariant 
cohomology ring with $\FF_2$-coefficients. To the fiber sequence
$$
X \hookrightarrow X_G \ra BG
$$
we associate the Leray-Serre spectral sequence:
\begin{equation}
\label{leray-serre}
E^{pq}_2 = H^p(G, H^q(X,\FF_2)) \longrightarrow H^{p+q}_G(X, \FF_2).
\end{equation}
The maximality of $(X,c)$ can be reformulated into the surjectivity of the restriction map from the 
equivariant cohomology to the usual cohomology  (see, for example, 
\cite[Chapter III, Proposition 4.16]{tomDieck}):
\begin{prop}
\label{coh-max}
Let $(X,c)$ be a topological space equipped with a cellular involution. The following conditions are equivalent:
\begin{itemize}
\item[ 1)] $X$ is an $M$-variety.
\item[ 2)] The action of $G$ on $H^{*}(X,\FF_2)$ is trivial and the Leray-Serre 
spectral sequence (\ref{leray-serre}) degenerates
at the second page.
\item[ 3)] The restriction homomorphism $R:H^{*}_{G}(X,\FF_2)
\rightarrow H^{*}(X,\FF_2)$ is surjective.
\end{itemize}
\end{prop}


\section{Cut-and-Paste construction of Hilbert squares over the reals}
\label{cut-paste}


Let $X$ be a nonsingular projective surface. By definition, the Hilbert square 
of $X,$ denoted by $X^{[2]},$ is the Hilbert scheme parametrizing 
$0$-dimen-sional subschemes of $X$ of length $2.$ As it follows from 
the universal property of the Hilbert schemes, $X^{[2]}$ has an elementary 
description (see, e.g., \cite[Example 7.3.1]{FGA}), which we recall for 
the convenience of the reader.

\smallskip

Let $\Bl_\Delta (X\times X)$ be the blow-up of  $X\times X$ along the diagonal 
$\Delta\subset X\times X.$ The involution $\tau$ on $X\times X$ permuting the 
factors lifts to an involution $\Bl(\tau)$ on  $\Bl_\Delta (X\times X).$ The quotient 
of $\Bl_\Delta (X\times X)$ by $\Bl(\tau)$  is naturally isomorphic to the Hilbert 
square $X^{[2]}.$ The fixed locus of $\Bl(\tau)$ is  the exceptional divisor of the 
blow-up $\Bl_\Delta (X\times X).$ The smoothness of this divisor implies that $X^{[2]}$ is
nonsingular.

The  branch locus $E\subset X^{[2]}$ of the double ramified covering 
$\Bl_\Delta (X\times X)\to X^{[2]}$ is naturally isomorphic to
the exceptional divisor of the blowup, and, since the normal vector bundle of  
$\Delta$ in $X\times X$ is isomorphic to the tangent vector bundle $TX$ of $X,$ 
both this exceptional divisor and $E$ are naturally isomorphic to 
$\PP(T^*X)$\footnote{In agreement with \cite{square}, which we use as reference, 
we follow the Grothendieck convention according to which the projectivization of a 
vector space is the space of its hyperplanes.}. In other words, 
a point of $E$ is identified with a point of $X$ plus a complex line in the tangent 
space at that point. Pairs of distinct points in $X$ represent the points of the complement, 
$X^{[2]}\setminus E.$

Notice that the normal bundle of $E$, as a branch locus, is the square of the 
normal bundle of the ramification locus. Hence, it is  naturally isomorphic to the 
square of the tautological bundle over $\PP(T^*X)$. 

The blowing-up map $\Bl_\Delta (X\times X)\ra X\times X$ descends to a morphism,
\begin{equation}
\label{H-C}
\pi: X^{[2]}\ra X^{(2)},
\end{equation}
where $X^{(2)}=(X\times X)/\tau$ is the $2$-fold symmetric product. This morphism 
is called the Hilbert-Chow map. Notice that the map $\pi$ is one to one on $X^{[2]}\setminus E$ and maps 
$E$ to the diagonal $\Delta X$ of $X^{(2)}.$ 

\smallskip

The construction of the Hilbert square is independent on the choice of the ground field. Applying it to a real
nonsingular projective surface $(X,\Conj ),$ it equips $X^{[2]}$ with a natural real structure, 
still denoted by $\Conj$. It acts on points in $X^{[2]}$ represented by pairs of distinct points 
in $X$ by sending a pair to the complex conjugate one, while if a point of $X^{[2]}$ is 
represented by a point of $X$ with a tangent line, then it is sent to the conjugate point 
with the conjugate line.

Thus, $X^{[2]}(\RR)=X/\Conj$ if $X(\RR)=\emptyset,$  and otherwise $X^{[2]}(\RR)$ 
is a smooth manifold with $\dim_\RR X^{[2]}(\RR)=4,$ having the following properties:

\begin{itemize}
\item $E(\RR)\subset X^{[2]}(\RR)$ is a submanifold of real codimension 1 in $X(\RR).$ 
It is diffeomorphic to 
$$
\PP_\RR(T^*X(\RR))=\bigsqcup_{i=1}^{r}\PP_\RR(T^*F_i),
$$
where $F_1,\dots, F_r$ are the connected components of $X(\RR).$

\item $X^{[2]}(\RR)\setminus E(\RR)$ is the disjoint union of the connected real 
$4$-dimensional manifolds  
$$
\left((X/\Conj) \setminus X(\RR)\right)\sqcup  \bigsqcup_{i=1}^r 
\left( F_i^{(2)}\setminus \Delta F_i\right) \sqcup \bigsqcup_{1\leq i<j\leq r}F_i\times F_j,
$$
where $\Delta F_i$ stands for the diagonal in $F_i^{(2)}.$ Each of the manifolds 
$F_i\times F_j$ in this decomposition is closed, while the manifolds $(X/\Conj) \setminus X(\RR)$
 and  $F_i^{(2)}\setminus \Delta F_i,\, i=1,\dots, r$ are open.

\item  For each $1\le i\le r$, the component $\PP_\RR(T^*F_i)$ of $E(\RR)$ is a 
common boundary of $ (X/\Conj) \setminus X(\RR)$ with  $F_i^{(2)}\setminus \Delta F_i$.
\end{itemize}

From this we can conclude  that $X^{[2]}(\RR)$ has the following decomposition, 
as illustrated in the figure below:
\begin{figure}[!h]
\label{figure-decomposition}
\begin{tikzpicture}[line cap=round,line join=round,>=triangle 45,x=0.7cm,y=0.7cm]
\clip(-6.5,-3) rectangle (20,6.5);
\draw [shift={(0,0)},line width=2.8pt]  plot[domain=0:3.141592653589793,variable=\t]({1*6*cos(\t r)+0*6*sin(\t r)},{0*6*cos(\t r)+1*6*sin(\t r)});
\draw [shift={(-5,0)},line width=2.8pt]  plot[domain=3.141592653589793:6.283185307179586,variable=\t]
({1*1*cos(\t r)+0*1*sin(\t r)},{0*1*cos(\t r)+1*1*sin(\t r)});
\draw [shift={(-3,0)},line width=2.8pt]  plot[domain=0:3.141592653589793,variable=\t]
({1*1*cos(\t r)+0*1*sin(\t r)},{0*1*cos(\t r)+1*1*sin(\t r)});
\draw [shift={(-1,0)},line width=2.8pt]  plot[domain=3.141592653589793:6.283185307179586,variable=\t]
({1*1*cos(\t r)+0*1*sin(\t r)},{0*1*cos(\t r)+1*1*sin(\t r)});
\draw [shift={(5,0)},line width=2.8pt]  plot[domain=3.141592653589793:6.283185307179586,variable=\t]
({1*1*cos(\t r)+0*1*sin(\t r)},{0*1*cos(\t r)+1*1*sin(\t r)});
\draw 
[shift={(1.9900569143071205,-1.58813075258282)},line width=2pt,dash pattern=on 1pt off 9pt]  
plot[domain=1.1634565649835678:1.9626857498870878,variable=\t]
({1*2.561646051862725*cos(\t r)+0*2.561646051862725*sin(\t r)},{0*2.561646051862725*cos(\t r)+1*2.561646051862725*sin(\t r)});
\draw [shift={(1.9900569143071205,-1.58813075258282)},line width=2pt,dash pattern=on 1pt off 3pt]  
plot[domain=1.9626857498870878:2.4680523563968904,variable=\t]({1*2.5616460518627244*cos(\t r)+0*2.5616460518627244*sin(\t r)},
{0*2.5616460518627244*cos(\t r)+1*2.5616460518627244*sin(\t r)});
\draw [shift={(1.9900569143071205,-1.58813075258282)},line width=2pt,dash pattern=on 1pt off 3pt]  
plot[domain=0.6686980247446463:1.1634565649835678,variable=\t]
({1*2.561646051862725*cos(\t r)+0*2.561646051862725*sin(\t r)},{0*2.561646051862725*cos(\t r)+1*2.561646051862725*sin(\t r)});
\draw [rotate around={0:(11.407825508008441,0)},line width=2.8pt] (11.407825508008441,0) ellipse (1.7cm and 0.8cm);
\draw [line width=2pt,color=ffqqqq] (-6,0)-- (-4,0);
\draw [line width=2pt,color=ffqqqq] (-2,0)-- (0,0);
\draw [line width=2pt,color=ffqqqq] (4,0)-- (6,0);
\draw [line width=2pt,dash pattern=on 1pt off 5pt,color=ffqqqq] (1,0)-- (3,0);
\begin{scriptsize}
\draw (7,0) node {$\sqcup$};
\draw (8,-0.2) node {$\displaystyle \bigsqcup_{1\leq i<j\leq r}$};
\draw (0,3) node {$\HH_0$};
\draw (-5,-0.5) node {$\HH_1$};
\draw(-1,-0.5) node {$\HH_2$};
\draw (5,-0.5) node {$\HH_r$};
\draw[color=ffqqqq] (5,0.3) node {$\PP_\RR(T^*F_r)$};
\draw[color=ffqqqq] (-1,0.3) node {$\PP_\RR(T^*F_2)$};
\draw[color=ffqqqq] (-5,0.3) node {$\PP_\RR(T^*F_1)$};
\draw(11.5,0) node {$F_i\times F_j$};
\draw (8.3,1) node {$X^{[2]}(\RR)_{\rm extra}$};
\draw  (-5,5) node {$X^{[2]}(\RR)_{\rm main}$};
\end{scriptsize}
\end{tikzpicture}
\caption{The  decomposition of $X^{[2]}(\RR)$}
\end{figure}

\vspace*{-0.2in}

$$
 X^{[2]}(\RR)=X^{[2]}_{\text{main}}(\RR) \bigsqcup X^{[2]}_{\text{extra}}
 (\RR),\quad  X^{[2]}_{\text{extra}}(\RR) = \bigsqcup_{1\leq i<j\leq r} \left(F_i\times F_j\right),
$$
where $X^{[2]}_{\text{main}}(\RR)$ is the connected component of  $X^{[2]}(\RR)$ that 
contains $E(\RR)$ in such a way that $E(\RR)$ divides $X^{[2]}_{\text{main}}(\RR)$ in
 $r+1$ connected submanifolds with boundary,
$$
X^{[2]}_{\text{main}}(\RR)= \bigcup_{i=0}^{r} \HH_i, 
$$
such that
\begin{align*}
\partial \HH_{0} = E(\RR),& \quad  \interior \HH_{0}\cong (X/\Conj) \setminus X(\RR) \\
\partial \HH_{i}=\PP_\RR(T^*F_i),& \quad  \interior \HH_{i}\cong F_i^{(2)}\setminus \Delta F_i, i=1,\dots, r.
\end{align*}
Each manifold
$\HH_i,\,  i=1,\dots,r$ is glued to $\HH_0$ along their common boundary
$\PP_\RR( T^*F_i )\subseteq \PP_\RR(T^*X(\RR)).$

\smallskip

Notice that the Hilbert-Chow map (\ref{H-C}) is equivariant with respect to the 
induced real structures on
 $X^{[2]}$ and  $X^{(2)}$.

\begin{prop}
\label{inclusion} 
Let $X$ be a real smooth projective surface with $X(\RR)\neq\emptyset.$ Then:
\begin{itemize}
\item[ 1)] For each $k=1,\dots, r$ the inclusion homomorphisms 
$$
\mathfrak{i}_{*}^ {\HH_k}: H_1(\PP_\RR T^* F_k)\to H_1(\HH_k)
$$ 
vanish at $\xi\in H_1(\PP_\RR T^* F_k)$ if and only if $\xi\cap\omega=0,$ where $\omega$ 
is the first Stiefel-Whitney class of the tautological line bundle over $\PP_\RR T^* F_k$.
\item[ 2)] 
The inclusion homomorphism
$$
\mathfrak{i}_{*}^ {\HH_0}:  H_1(\bigsqcup_{k=1}^r \PP_\RR T^* F_k)\to H_1(\HH_0)
$$ 
is non-vanishing on each $\xi\in H_1(\PP_\RR T^* F_k)$ with $\xi\cap\omega\ne 0.$ 
In particular, it is not vanishing on each of the fiber-classes $w_k\in  H_1(\PP_\RR T^* F_k).$ 
\item[ 3)]
If $X$ is maximal, then the homomorphism $\mathfrak{i}_{*}^ {\HH_0}$ is surjective.
\end{itemize}
\end{prop}

\begin{proof}
The given definition of $\omega$ is equivalent to saying that it is the characteristic 
class of the tautological double covering $UT^* F_k \to \PP_\RR T^* F_k$
(where $UT^*$ stands for the unit cotangent bundle). By the latter reason, we may 
also interpret $\omega$ as a restriction of the characteristic class $v_k$ of the double 
covering $(F_k\times F_k)\setminus \Delta F_k\to (F_k)^{(2)}\setminus \Delta F_k,$ 
or as a restriction of the characteristic class $v$ of the double covering 
$X\setminus X(\RR) \to (X/\Conj)\setminus X(\RR)$.
For the former one, it is sufficient to identify $UT^* F_k$ and  $\PP_\RR T^* F_k$ with the 
boundary of a tubular neighborhood of $\Delta F_k$ in $F_k\times F_k$ and $(F_k)^{(2)}$, 
respectively. Similarly, when it is a question of $X\setminus X(\RR) \to (X/\Conj)\setminus X(\RR)$,
we identify $\cup_k UT^* F_k$ and  $\cup_k \PP_\RR T^* F_k$ with the boundary of a tubular 
neighborhood of $\cup F_k$  in $X$ and $X/\Conj,$ respectively. For more details on extending the 
class $\omega$ to cohomology classes in $\HH_0$ and $\HH_i,\, i=1,\dots,r,$ we refer the 
interested reader to \cite[Section 3.1]{kr-jtop}.

Therefore, for each  
$\xi\in H^1(\HH_k)$ with $\xi\cap\omega=1,$ by the projection formula, we have
$$
\mathfrak{i}_*^{\HH_k}(\xi)\cap v_k=\xi\cap (\mathfrak{i}^{\HH_k})^*(v_k)=\xi\cap\omega=1.
$$
This implies the ``only if part" of the first statement and the second statement.

To prove the ``if part" of the first statement,  let $k\ge 1$ and consider a loop 
$\Gamma\subset \PP_\RR T^* F_k$ representing a class $\xi$ with $\xi\cap\,\omega=0$.
Such a loop lifts to a loop $\tilde \Gamma$ on the boundary of a tubular neighborhood of $\Delta F_k$ 
in $ F_k\times F_k$. Therefore, the homology class $[\tilde \Gamma] \in H_1(F_k\times F_k)$ belongs 
to the image of $H_1(\Delta F_k)$.  But, according to the K\"unneth formula and Poincar\'e duality, 
the class $[\tilde \Gamma],$ as every class in the image of $H_1(\Delta F_k),$ is of the form 
$\sum (x_j\otimes 1 + 1\otimes x_j),$ for some $x_j\in H_1(F_k).$ Since, in addition, the inclusion map 
$$
\inc^{\complement\Delta F_k}_*:H_1(F_k\times F_k\setminus \Delta F_k)\to H_1(F_k\times F_k)
$$ 
is an epimorphism and its kernel is generated by tubular circles $w$ around $\Delta F_k$, we may 
write the homology class $\tilde \xi$ realized in $H_1(F_k\times F_k\setminus \Delta F_k)$ 
by $\tilde \Gamma$ as $x+ \tau_* x + \epsilon w$, $\epsilon= 0,1$, for some
$x\in H_1(F_k\times F_k\setminus \Delta F_k)$ with $\tau$ being the deck transformation of 
our covering, which in the case under consideration is just transposition of the factors. 
Now, we are done:
\begin{align*}
\mathfrak{i}_*^{\HH_k}(\xi)= &\, (\pr_*\circ \inc^{\complement\Delta F_k}_*)(\tilde \Gamma)\\ 
= &\, \pr_*(x+\tau_* x +\epsilon w)\\
= &\, \pr_*(x)+\pr_*(\tau_* x)+ \epsilon \pr_*(w)\\
= &\, 0,
\end{align*}
where $\pr: F_k\times F_k\to F_k^{(2)}$ is the projection.

\smallskip

To prove the third statement, note, first, that $X(\RR)$ is a codimension 2 smooth submanifold of 
$X/\Conj,$ while $\bigsqcup_{k=1}^r \PP_\RR T^* F_k$ is the boundary of a tubular neighborhood of 
$X(\RR)$ in $X/\Conj$. Therefore, the surjectivity of $\mathfrak{i}_{*}^ {\HH_0}$ follows from the 
surjectivity of the inclusion homomorphism $H_1(X(\RR))\to H_1(X/\Conj ),$ which it in its turn follows 
from the exactness of the sequence  
$$
0\to H_1(X/\Conj, X(\RR))\to H_0(X(\RR))\to H_0(X)\to 0
$$ 
that holds due to maximality of $X$ (see Corollary \ref{maxSmith}).
\end{proof}

\begin{cor}
\label{inclusion-kernels}
For each $k=1,\dots, r,$ we have 
$\ker \mathfrak{i}_{*}^ {\HH_0}\subseteq \ker \mathfrak{i}_{*}^ {\HH_k}.$ 
\end{cor}

\proof
The inclusion $\ker \mathfrak{i}_{*}^ {\HH_0}\subseteq \ker \mathfrak{i}_{*}^ {\HH_k}$ 
follows from the first part of the second statement and the "if part" of the first statement 
of Proposition \ref{inclusion}.
\qed

\medskip

We collect next a few results which will be used several times in the next sections. For the convenience 
of the reader,  we indicate the main ideas of the proof.

\begin{prop}
\label{totaro-calc}
Let $X$ be a smooth compact complex surface. In an abbreviated notation $\beta_*=\beta_*(X)$ and 
$\beta_1=\beta_1(X),$ we have:
\begin{itemize}
\item[ 1)] The relation
\begin{equation}
\label{M-chi}
\chi(X^{[2]}(\RR))=\frac12\beta_*-2\beta_1+\frac12\chi(X(\RR))^2-\chi(X(\RR)).
\end{equation}
\item[ 2)] If $\Tors_2 H_*(X;\ZZ)=0,$ 
the relation
\begin{equation}
\label{totaro-notorsion}
\beta_*(X^{[2]}) = \frac12 \beta_*(\beta_*+1)+\beta_*- 2\beta_1.
\end{equation}
\item[ 3)] If $\Tors_2 H_*(X;\ZZ)\neq 0,$ the relation 
\begin{equation}
\label{totaro-torsion}
\beta_*(X^{[2]}) \geq \frac12 \beta_*(\beta_*+1)+\beta_*- 2\beta_1.
\end{equation}
\end{itemize}
\end{prop}

\begin{proof}
The first item follows, for example, from the more general formula computing the Euler characteristic of the 
Hilbert scheme of points of surfaces defined over the reals \cite[page 5452]{kr}, and the observation that 
$\chi(X)=\beta_*(X)-4\beta_1(X).$

The other two items are implicitly contained in the proof of  Theorem 2.2 in \cite{square}.
According to \cite[Theorem 2.2]{square}, $\Tors_2 H_*(X^{[2]}; \ZZ)=0$ as soon as $\Tors_2 H_*(X;\ZZ)=0$.
Therefore, both parts of the statement follow from the following relation for the ordinary Betti numbers
\begin{equation}
\label{cx-relation}
\dim \, H_*(X^{[2]}; \QQ) = \frac12 b_*(b_*+1)+b_*- 2b_1
\end{equation}
where 
$$
b_i=\dim \, H_i(X;\QQ), \quad b_*=\sum b_i.
$$
This relation follows from the presentation of $X^{[2]}$ as the quotient of the blowup 
$\Bl_\Delta (X\times X)$ of $ X\times X$ due to the following arguments:
\begin{enumerate}
\item $H_*(X^{[2]};\QQ)$ is canonically (by pull-back) isomorphic to the 
$\Bl(\tau)$  invariant subspace of $H_*(\Bl_\Delta (X\times X);\QQ)$.
\item $H_*(\Bl_\Delta (X\times X);\QQ)$ splits canonically as 
$$H_*(\Bl_\Delta (X\times X);\QQ)=H_*(X\times X;\QQ)\oplus H_*(\Delta(X\times X); \QQ).$$
\item $\Bl(\tau)$ acts identically on the second summand, while on the first summand it maps 
the basic elements $v_i\otimes v_j$, $v_i\in H_i(X\times X;\QQ), v_j\in H_j(X\times X;\QQ)$
to $(-1)^{ij} v_i\otimes v_j$.
\end{enumerate}
\end{proof}
\begin{rmk}
{\rm 
A different proof of the relation (\ref{cx-relation}) follows 
from a simple inspection of G\"ottsche's formula \cite{gottsche} for the 
Betti numbers of the Hilbert scheme of points on smooth projective surfaces.
}
\end{rmk}


\section{Proof of Theorem \ref{converse}}


As a first step, we show the following:

\begin{prop}
\label{nonemptyness}
Let $X$ be a real smooth projective surface. If $X^{[2]}$ is maximal, then $X(\RR)\neq\emptyset.$
\end{prop}

\begin{proof} 
By contradiction,  let assume that $X^{[2]}$ is maximal and $X(\RR)=\emptyset$. Then $X^{[2]}(\RR)$ is 
the quotient smooth $4$-manifold $X/\Conj,$ and so 
$$
\beta_*(X^{[2]}(\RR))=\beta_*(X/\Conj).
$$
Since $X$ and $X/\Conj$ are connected, by applying the homology Smith sequence we find
$$
\cdots \ra H_1(X)\ra H_1(X/\Conj)\ra H_0(X/\Conj) =\FF_2,
$$
and so $\beta_1(X/\Conj)\leq \beta_1+1$ where $\beta_i$ states for $\beta_i(X).$ From Poincar\'e 
duality and the Riemann-Hurwitz formula we find 
\begin{align*}
\beta_*(X/\Conj)=\, &\,\chi(X/\Conj)+4\beta_1(X/\Conj)\\
\leq &\, \frac12\, \chi(X)+4\beta_1+4\\
= &\,\frac12\,\beta_*+2\beta_1+4.
\end{align*}
Thus,
$$
\beta_*(X^{[2]}(\RR))= \beta_*(X/\Conj)\leq \frac12\beta_*+2\beta_1+4=5+3\beta_1+\frac12\beta_2.
$$
Using the estimate (\ref{totaro-torsion}) in Proposition \ref{totaro-calc}, we notice now that 
\begin{align*}
\beta_*(X^{[2]})\geq &\, \frac12 \beta_*(\beta_*+1)+\beta_*- 2\beta_1\\
=\,&\,\frac12(2+2\beta_1+\beta_2)^2+ 3+ \beta_1+\frac32\beta_2\\
>\,&\,5+3\beta_1+\frac12\beta_2\\
\ge \,&\,\beta_*(X^{[2]}(\RR)),
\end{align*}
contradicting the maximality of $X^{[2]}.$ 
\end{proof}

\begin{proof}[Proof of Theorem \ref{converse}]
Pick a point $p\in X(\RR),$ whose existence is ensured by Proposition \ref{nonemptyness}, 
and consider the map $f:X\ra X^{(2)}$ given by 
$$
f(x)=\{p,x\}.
$$

Since the Hilbert-Chow map $\pi:X^{[2]}\ra X^{(2)}$ is an isomorphism when restricted to 
$X^{[2]}\setminus E$ and $f(X\setminus\{p\})\cap \pi(E)=\emptyset,$ the restriction of $f$ to $X\setminus\{p\}$ 
induces a map 
$$
\phi:X\setminus \{p\}\rightarrow X^{[2]}.
$$
The map $\phi$ extends to the blowup $\Bl_p(X)$ of $X$ at the point 
$p,$ and so we have a commutative diagram
$$
\xymatrix{\Bl_p(X)\ar[r]^\phi\ar[d]_{\pr}&X^{[2]\ar[d]^{\pi}}\,\\
 X
 \ar[r]_f &X^{(2)}.}
$$

\begin{lem}\label{surjectivity}
The map $\phi^*:H^*(X^{[2]})\ra H^*(\Bl_p(X))$ is surjective.
\end{lem}

\begin{proof} Note, first, that $H^i(\Bl_p(X))=H^i(X)\oplus H^{i}(\PP^1)$ for any $i>0$.
Here, the first summand coincide with $\im\pr^*$. Since $f^*:H^*(X^{(2)})\ra H^*(X)$ is 
surjective (see, for example, \cite[Lemma 2.5]{franz} for a more general statement), 
this implies that the first summand is contained in $\im\phi^*$.

The second summand is nontrivial only for $i=2$. It is generated by $\FF_2$-reduction 
of the integer class $\inc^*c_1(\theta),$ where $\theta$ is the tautological line bundle over 
$\PP(T^*X)=E\subset X^{[2]}$ and $\inc: \PP^1\hookrightarrow E$ is the inclusion. Since 
$c_1(\theta)$ is the restriction of an appropriate element $e\in H^2(X^{[2]};\ZZ)$ 
(see, for example, \cite[page 4]{square}), the second summand is also contained $\im\phi^*$.
\end{proof}

To finish the proof of Theorem \ref{converse}, consider now the commutative diagram
$$
\xymatrix{&H_G^*(X^{[2]})\ar[r]\ar[d]_{R^{[2]}} &H_G^*(\Bl_p(X))\,\ar[d]^R\\
&H^*(X^{[2]})\ar[r]^{\phi^*}
&H^*(\Bl_p(X)).\\
}
$$
Since $X^{[2]}$ is maximal the  induced restriction map $R^{[2]}$ is onto, which together 
with the surjectivity of $\phi^*$ implies that the restriction map $R$ is onto, as well. Hence, 
by Proposition \ref{coh-max}, $X$ is maximal.
\end{proof}


\section{Proof of Theorem \ref{main} and Corollaries \ref{main-corollary}, \ref{rational}}


\subsection{Reduction of the proof to computation of $\beta_1(X^{[2]}(\RR))$}
\label{reduction}


As before, let $\beta_i,\, i=0,\dots,4$ denote the $\FF_2$ - Betti numbers of $X.$ Under assumption of 
Theorem \ref{main}, we have $\beta_1=\beta_3=0$, while $\beta_2$ coincides with the ordinary Betti 
number $b_2=\dim H_2(X;\QQ)$.  

\smallskip

According to Proposition \ref{totaro-calc}, we find 
\begin{align*}
\beta_*(X^{[2]})& = 5+3\beta_2+\frac12 \beta_2(\beta_2+1),\\
\chi(X^{[2]}(\RR))& = 1+\frac12 \beta_2+\frac12\chi(X(\RR))^2-\chi(X(\RR)).
\end{align*}

Since 
$\displaystyle \beta_1(X^{[2]}(\RR))=
\frac14\left[\beta_*(X^{[2]}(\RR))-\chi(X^{[2]}(\RR))\right],$ 
the above two relations imply that $X^{[2]}$ is maximal if and only if 
\begin{align}
\label{b1-max-raw}
\beta_1(X^{[2]}(\RR))= & \frac14\left[\beta_*(X^{[2]})-\chi(X^{[2]}(\RR))\right]\\ \notag
=& \frac14\left[4+3\beta_2+\frac12 \beta_2^2+\chi(X(\RR))-\frac12\chi^2(X(\RR))\right].
\end{align}
It is convenient to rewrite this expression in terms of the number of the connected components of $X(\RR),$ 
which will be denoted by $r.$ For that, we notice, first, that the maximality of $X(\RR)$ is equivalent to
$$
2r+\beta_1(X(\RR))=2+\beta_2, 
$$
which implies  $\chi(X(\RR))=2r-\beta_1(X(\RR))=4r-2-\beta_2.$  Then, an immediate computation transforms 
the criterium (\ref{b1-max-raw}) into 
\begin{equation}
\label{b1-max}
\beta_1(X^{[2]}(\RR))=3r-2r^2+r\beta_2.
\end{equation}


\subsection{Direct computation of $\beta_1(X^{[2]}(\RR))$}
\label{direct}


The cut-and-paste construction described in Section \ref{cut-paste} is used next to compute 
$\beta_1(X^{[2]}(\RR)).$ We follow the notations introduced therein.

\begin{lem}
\label{b1-extra}
$\displaystyle \beta_1(X^{[2]}_{\emph{extra}}(\RR))=r\beta_2-2r^2+4r-\beta_2-2.$
\end{lem}

\begin{proof}
Using the  K\"unneth formula, we get
\begin{align}
\label{extra-general}
\beta_1(X^{[2]}_{\text{extra}}(\RR))=&\sum_{i<j}(\beta_1(F_i)+\beta_1(F_j))\\ \notag
=&\sum_{i<j}(\beta_*(F_i)+\beta_*(F_j)-4)\\ \notag
=&\frac12 \sum_{i,j}(\beta_*(F_i)+\beta_*(F_j)-4) - \frac12 \sum_{i}(2 \beta_*(F_i) -4)\\ \notag
= &\,r\beta_*(X(\RR))-2r^2-\beta_*(X(\RR))+2r.
\end{align}
Due to the maximality of $X(\RR),$ we have
$$
\beta_*(X(\RR))=\beta_*(X)=2+\beta_2,
$$
and  the proof of the lemma follows from (\ref{extra-general}).
\end{proof}

\begin{lem}
\label{b1-main} 
$\beta_1(X_{\emph{main}}^{[2]}(\RR))=3+\beta_2-2r.$
\end{lem}
\begin{proof}
We start by computing $\beta_1(\HH_0).$ By the Poincar\'e-Lefschetz duality we have 
$\beta_1(\HH_0)=\beta_3(X/\Conj, X(\RR)).$ On the other hand, using the Smith exact sequence
\begin{align*}
0\to & H_4(X/\Conj, X(\RR)) \xrightarrow{\tr^*} H_4(X)=\FF_2\xrightarrow{\pr_*} H_4(X/\Conj, X(\RR)) 
\xrightarrow{\Delta} \\
&  \xrightarrow{\Delta} H_3(X/\Conj, X(\RR)) \xrightarrow{\tr^*} H_3(X)=0.
\end{align*}
we conclude that $H_4(X/\Conj, X(\RR))=\FF_2$ and the map $\Delta$ is an isomorphism. Therefore, 
$\beta_1(\HH_0)=1.$

For $\beta_1(\HH_i)$ with $i=1,\dots,r$, once more by Poincar\'e-Lefschetz duality, we have
$\beta_1(\HH_i)=\beta_1(F_i^{(2)}\setminus \Delta F_i)=\beta_3(F_i^{(2)},\Delta F_i).$ 
On the other hand, the diagonal $\Delta F_i\hookrightarrow F_i^{(2)}$ vanishes in $H_2(F_i^{(2)})$ as
the $\FF_2$-homology class realized by $\Delta F_i$ in $H_2(F_i^{(2)})$ is the boundary of 
the Poincar\'e-Lefschetz dual to the characteristic class of the unramified covering 
$F_i\times F_i\setminus \Delta F_i\to F_i^{(2)}\setminus \Delta F_i $.
Thus, the following short exact sequence holds
$$
0\ra H_3(F_i^{(2)})\ra H_3(F_i^{(2)},\Delta F_i)\ra H_2(\Delta F_i)\ra 0,
$$
and implies $\beta_3(F_i^{(2)},\Delta F_i)=\beta_3(F_i^{(2)})+1.$ 
Furthermore, by applying the Macdonald formula 
\cite[Proposition 4.2]{macdonald} expressing the  Betti numbers 
of symmetric products when $F_i$ is oriented, and  the Kallel-Salvatore 
formula \cite[Proposition 20]{ss} (see \cite[Theorem 3.3]{baird} 
for a more precise formulation) in the non-orientable case, we find 
$\displaystyle \beta_3(F_i^{(2)})=\beta_1(F_i).$ Summing-up, we find 
$\beta_1(\HH_i)=1+\beta_1(F_i)=\beta_*(F_i)-1.$

Thus, we obtain
\begin{align}
\label{pieces-sum}
\sum_{i=0}^r \beta_1(\HH_i)= &\,1+\sum_{i=1}^r(\beta_*(F_i)-1)\\ \notag
= &\, 1+\beta_*(X(\RR))-r\\ \notag
= &\, 3+\beta_2-r.
\end{align}

To finish the proof we apply the Mayer-Vietoris sequence to the decomposition $X^{[2]}=A\cup B$ with $A=\HH_0$
and $B=\sqcup_{i=1}^r\HH_i:$
$$
\bigoplus_{i=1}^r H_1(\PP_\RR T^* F_i)\xrightarrow{\mu}\bigoplus_{i=0}^r H_1(\HH_i)
\ra H_1(X^{[2]}_{\text{main}}(\RR))\ra 0
$$
and calculate the rank of $\mu$ as follows.

\begin{lem}
\label{rank-mu}
$\displaystyle \rank \mu= r.$
\end{lem}

\begin{proof}
According to Proposition \ref{inclusion}, for each $i=1,\dots, r$  the inclusion map
$$
\mathfrak{i}^{\HH_i}_* : H_1(\PP_\RR T^* F_i)\to H_1(\HH_i)
$$
vanishes on $\xi\in H_1(\PP_\RR T^* F_i)$ if and only if $\xi\cap\omega=0,$ where 
$\omega$ is the first Stiefel-Whitney class of the tautological line bundle over $\PP_\RR T^* F_i.$
Since $\omega\ne 0$ in $H^1(\PP_\RR T^* F_i)$ for each $i=1,\dots, r$, 
this implies $\rank \mathfrak{i}^{\HH_i}_*=1$.

Under the assumption that $\beta_1(X)=0,$ we can see that a similar conclusion holds for $i=0,$ 
as well. Namely, since $H_1(X)=0,$ we find that $H_1(\HH_0)=H_1((X/\Conj) \setminus X(\RR))=\FF_2$ 
and the only its nontrivial element is represented by any of the tubular circles $w_i$ around $F_i\subset X(\RR).$
Since for each of them, $w_i\cap\omega=1$, by applying the second statement of Proposition \ref{inclusion} we 
conclude that the inclusion map 
$$
\mathfrak{i}^{\HH_0}_* :  H_1(\bigsqcup_{i=1}^r \PP_\RR T^* F_i)=
\bigoplus^r_{i=1} H_1(\PP_\RR T^* F_i)\to H_1(\HH_0)
$$
vanishes on $\xi\in H_1(\PP_\RR T^* F_i)$ if and only if $\xi\cap\omega=0.$

Therefore, the kernel of $\mathfrak{i}^{\HH_0}_*$ coincides with the kernel of  
$\bigoplus_{i=1}^r\mathfrak{i}_{*}^ {\HH_i}.$ Since 
$$
\rank \bigoplus_{i=1}^r\mathfrak{i}_{*}^ {\HH_i}=\sum^r_{i=1} \rank \mathfrak{i}_{*}^ {\HH_i}= r,
$$
we conclude that $\rank \mu= \rank (\mathfrak{i}^{\HH_0}_* \oplus \bigoplus_{i=1}^r\mathfrak{i}_{*}^ {\HH_i})$ 
is equal to $r,$ too.
\end{proof}

From $(\ref{pieces-sum})$ and Lemma \ref{rank-mu} we deduce  
\begin{equation*}
\beta_1(X_{\text{main}}^{[2]}(\RR))=\sum_{i=0}^r \beta_1(\HH_i)-r=3+\beta_2-2r,
\end{equation*}
which concludes the proof of Lemma \ref{b1-main}.
\end{proof}
\subsection{End of the proof of Theorem \ref{main}}
As we proved in Section \ref{reduction},  $X^{[2]}$ is maximal if and only if 
$\beta_1(X^{[2]}(\RR))=3r-2r^2+r\beta_2.$ On the other hand, combining the results of Lemmas 
\ref{b1-extra} and \ref{b1-main} we get
$$
 \beta_1(X^{[2]}(\RR))=1+r\beta_2+2r-2r^2.
$$
Obviously, the two results agree if and only if $r=1$. 
\qed

\begin{rmk}
{\rm 
The computations performed during the proof of Theorem \ref{main} give, in fact, a control 
not only on the value of $\beta_1(X^{[2]}(\RR))$, but over the value of all Betti numbers of  $X^{[2]}(\RR).$ 
We find 
\begin{align}\notag
\beta_0(X^{[2]}(\RR))=&\beta_4(X^{[2]}(\RR))=\frac12r(r-1)+1,\\ \notag
\beta_1(X^{[2]}(\RR))=&\beta_3(X^{[2]}(\RR))=r\beta_*+1-2r^2,\\ \notag
\beta_2(X^{[2]}(\RR))=&\frac12\beta_*(\beta_*-1)-2(r-1)\beta_*+3r(r-1). 
\end{align}
In particular, when $X$ satisfies assumptions of Theorem \ref{main},
the defect of maximality of $X^{[2]}(\RR),$ 
that is the difference 
$\beta_*(X^{[2]}) - \beta_*(X^{[2]}(\RR))$, is equal to $4(r-1).$ 
}
\end{rmk}


 \subsection{Proof of Corollaries \ref{main-corollary}, \ref{rational}}


Both corollaries are immediate consequences of Theorem \ref{main} plus, in the case of Corollary 
\ref{main-corollary}, the following statement.

\begin{lem}
\label{hodge-positive}
If $ H_1(X;\ZZ)$ has no 2-torsion and $h^{2,0}(X)+h^{1,0}(X) >0$, then
$X$ does not admit a maximal real structure with connected $X(\RR)$.
\end{lem}
 
 \begin{proof} Due to absence of 2-torsion, the maximality relation can be rewritten in 
 terms of ordinary Betti numbers as follows:
 $$
 2b_0(X(\RR))+b_1(X(\RR))= b_*(X) = \sum h^{p,q}(X).
 $$
 On the other hand, by Comessatti inequality (see, f.e., \cite{dk}), 
 $$
 2 - (2b_0(X(\RR))-b_1(X(\RR)))\le h^{1,1}(X).
 $$
Combining them together we get $b_0(X(\RR))\ge 1+\frac12 h^{2,0}+h^{1,0},$ which implies 
$b_0(X(\RR))\ge 2.$
 \end{proof}


 \section{Proof of Theorem \ref{b1>0}}


We follow the same strategy as in the proof of Theorem \ref{main}, and the same notations.

\smallskip

Let $X$ be a real smooth projective surface and consider the 
Mayer-Vietoris sequence
\begin{equation}
\label{MV-surf}
\bigoplus_{i=1}^r H_1(\PP_\RR T^* F_i)\xrightarrow{\mu}\bigoplus_{i=0}^r H_1(\HH_i)
\ra H_1(X^{[2]}_{\text{main}}(\RR))\ra 0.
\end{equation}

\begin{prop}\label{iff}
\label{gen-max-crit}
Let $X$ be a maximal surface with $\Tors_2 H_1(X(\CC), \ZZ)=0$. Then $X^{[2]}$ is 
maximal if and only if $\rank \mu =1+\beta_1.$
\end{prop}

\proof
Notice first that the maximality of $X(\RR)$ is equivalent to $\chi(X(\RR))= 4r -\beta_*.$ 
By applying Proposition \ref{totaro-calc} to the maximal surface $X$ we conclude that $X^{[2]}$ 
is maximal if and only if we have
\begin{align}
\label{beta1+b1}
\beta_1(X^{[2]}(\RR))=&\,\frac14\left[\frac12 \beta_*^2+\beta_*+\chi(X(\RR))-\frac12\chi(X(\RR))^2\right] \\\notag
=&\, r\beta_*-2r^2 +r.
\end{align}
Furthermore,  using (\ref{extra-general}) we find that (\ref{beta1+b1}) is equivalent with  
$$
\beta_1(X^{[2]}_{\text{main}}(\RR))=\beta_*-r.
$$
On the other hand, as in the proof of Lemma \ref{b1-main}, from 
Corollary \ref{maxSmith} and (\ref{pieces-sum}), we 
deduce that
\begin{equation}
\label{betaH0}
\beta_1(\HH_0)=1+\beta_1,\quad \sum^r_{i=1} \beta_1(\HH_i)=\beta_* -r.
\end{equation}
From ({\ref{MV-surf}) we infer now that $X^{[2]}$ is maximal if and only if $\rank \mu =1+\beta_1.$
\qed

According to Proposition \ref{gen-max-crit}, to finish the proof of Theorem \ref{b1>0},
it remains to prove the following lemma.

\begin{lem}
\label{rank-mu-notorsion}
Under assumptions of Theorem \ref{b1>0} we have:
\begin{itemize}
\item[ 1)] $\rank \mu= 1+\beta_1$ if $X(\RR)$ is connected.
\item[ 2)] $ \rank \mu \ge  2+\beta_1$ if $r>1+\beta_1$.
\end{itemize}
\end{lem}

\begin{proof}
If $X(\RR)$ is connected (that is $r=1$), then, according to Corollary 
\ref{inclusion-kernels}, $\ker \mathfrak{i}_*^{\HH_0}$ is contained in 
$\ker \mathfrak{i}_*^{\HH_1}$. Therefore, 
$\rank \mu= \rank (\mathfrak{i}_*^{\HH_0}\oplus 
\mathfrak{i}_*^{\HH_1})=\rank \mathfrak{i}_*^{\HH_0},$ 
while by (\ref{betaH0}) and Proposition \ref{inclusion}.(3) we have 
$\rank \mathfrak{i}_*^{\HH_0}=\beta_1(\HH_0)=1+\beta_1$.

\smallskip

Assume now that $r>1+\beta_1.$ Under this assumption, the homomorphism, 
we  show first that 
$\mathfrak{i}_*^{\HH_0} : H_1(\bigsqcup_{i=1}^r \PP_\RR T^* F_i)\to H_1(\HH_0)$ 
vanishes on at least one non-trivial linear combination of the fiber-classes 
$w_i, \, i=1,\dots, r.$

Pick a point $p_i$ on each $F_i\subset X$ with $i=1,\dots, r$. For each $i=2\dots, r,$ 
join $p_1$ with $p_i$ by a generic path $\gamma_i$ in $X$ and consider the circles 
$C_i$ formed by $\gamma_i\cup \Conj(\gamma_i).$  Since their number, $r-1,$ is strictly 
larger than $\beta_1(X),$ there exists a 2-dimensional submanifold $\Upsilon\subset X$ 
bounding some nontrivial combination of circles $C_i.$ 
For convenience, we choose $\Upsilon$ such that it is  transversal to $\sqcup_{i=1}^r F_i.$ 
Its image $\pr\Upsilon$ under the projection map $\pr:X\ra X/\Conj$ is a $2$-cycle, 
giving rise to an element  $[\pr\Upsilon]\in H_2(X/\Conj).$  
For each $i$ such that $C_i$ is a part of $\partial \Upsilon,$ let $q\in F_i\cap \Upsilon$ and 
compute the local intersection number of $\pr \Upsilon$ and $F_i$ at $\pr(q)$ in $X/\Conj.$ 
By the projection formula, we have 
\begin{align}
\label{LocalIntersection}
\left([\pr\Upsilon]\circ F_i\right)_{\pr(q)}=&\frac12\left(\pr^*[\pr \Upsilon]\circ \pr^*F_i\right)_q\notag \\
=&\left(\pr^*[\pr \Upsilon]\circ F_i\right)_q\notag \\
=&\left([\Upsilon\cup\Conj(\Upsilon)]\circ F_i\right)_q.
\end{align}
If $q=p_i,$ the cycle $\Upsilon\cup \Conj(\Upsilon)$ intersects $F_i$ 
transversally at the point $p_i.$ By (\ref{LocalIntersection}), we find that the 
local intersection number $\left([\pr\Upsilon]\circ F_i\right)_{\pr(q)}$ 
equals to $1.$ If $q \neq p_i,$ then each
$\Upsilon$ and $\Conj(\Upsilon)$ intersects $F_i$
transversally at $q$, and so $([\Upsilon\cup\Conj(\Upsilon)]\circ F_i)_q=2$.
We conclude that the cycle $\pr\Upsilon$ 
has  a non-trivial $\FF_2$-intersection with $F_i\subset X/\Conj$ for some $i\in\{2\dots, r\}.$ 
This implies that there exists at least one $i\in \{2,\dots, r\}$ such that $F_i$ represents a 
non-zero element of $H_2(X/\Conj).$

To get a nontrivial linear combination of fiber-classes mapped to zero by 
$$
\mathfrak{i}_*^{\HH_0}: H_1(\bigsqcup_{i=1}^r \PP_\RR T^* F_i)\to H_1(\HH_0),
$$ 
we identify $\bigsqcup_{i=1}^r \PP_\RR T^* F_i$ with the boundary of a tubular 
neighborhood $U$ of $\cup_{i=1}^rF_i$ in $X/\Conj$, and $\HH_0$ with
$X/\Conj\setminus \interior U$. Next, we pick  $F_k$ representing a non-zero 
element in $H_2(X/\Conj)$ and  take a generic smooth $2$-dimensional 
submanifold $\Sigma\subset X/\Conj$ which has non-trivial $\FF_2$-intersection 
with this $F_k.$ Then, the linear combination of the fiber classes 
given by $\Sigma\cap \partial U\subset \partial U=\bigsqcup_{i=1}^r \PP_\RR T^* F_i$ 
is bounded by $\Sigma\setminus \interior U$ and, thus, vanishes in $H_1(\HH_0)$.

According to the first item of Proposition \ref{inclusion}, any non-trivial combination of 
fiber-classes is not mapped to zero by $\bigoplus_{i=1}^r\mathfrak{i}_{*}^ {\HH_i}.$ 
Therefore, we conclude that 
$\rank \mu= \rank (\mathfrak{i}_*^{\HH_0}\bigoplus_{i=1}^r\mathfrak{i}_{*}^ {\HH_i})
>\rank \mathfrak{i}_*^{\HH_0}.$
Finally, there remains to notice that, by (\ref{betaH0}) and the last item of Proposition 
\ref{inclusion}, we have $\rank \mathfrak{i}_*^{\HH_0}=1+\beta_1.$
\end{proof}

\begin{rmk}
{\rm 
When $X$ is maximal and $\Tors_2 H_1(X(\CC), \ZZ)=0,$ from the above 
considerations we find that the defect of maximality of $X^{[2]}(\RR)$ is  
$$
\beta_*(X^{[2]}) - \beta_*(X^{[2]}(\RR))=4(\rank \mu-\beta_1-1). 
$$
In particular, according to the second item of Lemma \ref{rank-mu-notorsion},
we find $\beta_*(X^{[2]}) - \beta_*(X^{[2]}(\RR))\ge 4$ 
as soon as  $r>\beta_1+1$.
}
\end{rmk}


\section{Examples and applications}


We start by presenting several examples, emphasizing the results obtained in Theorems \ref{converse}, 
\ref{main} and \ref{b1>0}.


\subsection{Examples}


Theorem \ref{converse} provides a quick way to produce examples of Hilbert squares which are not maximal, 
and examples are abundant. Here we would like to draw the reader's attention to a special phenomenon. 
There exist complex deformation classes of surfaces that contain surfaces with real structure, but none of 
these real structures is maximal. This is the case, for example, of Campedelli surfaces 
\cite{Campedeli} and of Miayoka-Yau surfaces with $h^{2,0}\le 3$ \cite{Kh-Ku}. Such surfaces cannot have 
maximal Hilbert square.

 \smallskip
 
 We will address next the Hilbert square of maximal surfaces only.

\subsubsection{Elliptic surfaces}

For simplicity, we restrict ourselves to minimal elliptic surfaces with rational base and without multiple fibers. 
Over the complex field, such elliptic surfaces form a countable number of deformation classes determined just 
by the Euler characteristic $\chi(X)=12k$, $k\in \ZZ_{\ge 1}$. Such elliptic surfaces satisfy $H_1(X;\ZZ)=0.$  
A straightforward application of Theorem \ref{main} and Corollary \ref{main-corollary} yields the following answer:

\begin{prop}
\label{elliptic}
For $k=1$, there exist maximal real elliptic surfaces $X$ with $\chi(X(\RR))=-8$ and connected $X(\RR)$; 
for these surfaces $X^{[2]}$ is maximal. For each $k\ge 2$, all maximal real elliptic surfaces $X$ have 
disconnected $X(\RR)$; for these surfaces $X^{[2]}$ 
is never maximal.
\end{prop}

\begin{proof} For $k=1$, it is sufficient to pick two real plane cubics intersecting each other 
at $9$ pairwise distinct real points and to take as $X$ the blowup of $\PP^2$ at these $9$ points. 
Then the maximality of $X^{[2]}$ follows from Theorem \ref{main}.

If $k\ge 2$, it is sufficient to observe that $h^{2,0}(X)>0$. After that there remains to apply 
Corollary \ref{main-corollary}.
\end{proof}

\subsubsection{Ruled surfaces}

As a byproduct of the proof of Theorem \ref{b1>0}, we obtain the following result:

\begin{thm}\label{ruled} 
Let $C$ be a maximal curve, and $\cal E$ a rank $2$ complex vector bundle with a 
real structure which lifts the real structure of $C$. Then, the Hilbert square $X^{[2]}$ 
of the ruled surface $X=\PP_C(\cal E)$ is maximal.
\end{thm}

\begin{proof} 
As it was established in Proposition \ref{gen-max-crit}, $X^{[2]}$ is maximal if and only if 
$\rank \mu =1+\beta_1.$ Recall now that, from Corollary \ref{maxSmith} we find 
$\beta_1(\HH_0)=1+\beta_1$, while, according to the third item of Proposition \ref{inclusion}, 
the map 
$$
\mathfrak{i}_{*}^ {\HH_0}:  H_1(\bigsqcup_{k=1}^r \PP_\RR T^* F_k)\to H_1(\HH_0)
$$ 
is surjective. Therefore, to prove that $\rank \mu = \rank \mathfrak{i}_{*}^ {\HH_0},$ it is sufficient 
to show that $\ker(\mathfrak{i}_{*}^ {\HH_0})\subseteq \ker(\bigoplus_{i=1}^r\mathfrak{i}_{*}^ {\HH_i}).$

For that we notice that each $F_i$ can be seen as a cooriented fibration with fiber $S^1$ 
over a real component $C_i$ of $C.$ Hence, the $2$-cycle formed in $\PP_\RR T^* F_i$ 
by the line directions tangent to the fibers of this fibration is Poincar\'e dual to the 
Stiefel-Whitney class $\omega$  of the tautological line bundle over  $\PP_\RR T^* F_i.$ 
Furthermore, identifying $\PP_\RR T^* F_i$ with the boundary of a tubular neighborhood 
of $F_i$ in $X/\Conj,$ we observe that this $2$-cycle is cut on $\PP_\RR T^* F_i$ by a 
$3$-chain $\zeta_i$ formed by the discs bounding the above $S^1$-fibers, the discs given 
by the quotient $\PP_C(\cal E)/\Conj$ over $C_i.$ Let now 
$\displaystyle x=\sum_{i=1}^r x_i\in H_1(\bigsqcup_{k=1}^r \PP_\RR T^* F_k)$ such that 
$(\bigoplus_{i=1}^r\mathfrak{i}_{*}^ {\HH_i})(x)\neq 0.$ Then, $x\cap \omega \neq 0,$ 
and so there exists $i\in\{1,\dots,r\}$ such that $x_i\cap \omega\ne 0.$ That implies 
$x\cap \zeta_i\neq 0,$ and so $\mathfrak{i}_{*}^ {\HH_0}(x)\neq 0.$
 \end{proof}

When $C$ is of genus $g\ge 1$, the ruled surfaces $X=\PP_C(\cal E)$  provide examples of maximal 
surfaces with torsion free homology, whose real locus is disconnected 
(namely, $1<\beta_0(X(\RR))=g+1\le \beta_1(X)=2g$), but whose Hilbert square is maximal.

\subsubsection{Abelian surfaces}

Contrary to the case of ruled surfaces, for real abelian surfaces we obtain the following result.

\begin{thm}
No real abelian surfaces has maximal Hilbert square.
\end{thm}

\begin{proof}
It is sufficient to consider the case of maximal surfaces and to check that 
$\rank \mu > 1+\beta_1$. But, in such a case, the action of $\Conj $ on $X$ 
is diffeomorphic to the action of $(-1)\times id$ on $E\times E$ where $E$ is an elliptic curve. 
Thus, $\HH_0$ is diffeomorphic to the product of $E$ with a 4-holed sphere.

This shows that the sum of the 4 fiber-classes $w_k\in H_1(P_\RR T^*F_k)$ belongs to 
$\ker \mathfrak{i}_*^{\HH_0}$ and that $\rank \mathfrak{i}_*^{\HH_0}=1+\beta_1$. 
On the other hand, by Proposition \ref{inclusion}, neither of nontrivial combinations of $w_k$
is mapped to $0$ by $\bigoplus_{i=1}^4\mathfrak{i}_{*}^ {\HH_i}$. Therefore,
$\rank \mu= \rank (\mathfrak{i}_*^{\HH_0}\bigoplus_{i=1}^4\mathfrak{i}_{*}^ {\HH_i})
>\rank \mathfrak{i}_*^{\HH_0}=1+\beta_1.$
\end{proof}

For comparison with the case of ruled surfaces, let us note that if $X$ 
is a maximal abelian surface, 
we have $1<\beta_0(X(\RR))=\beta_1(X)=4.$

\subsubsection{Rational surfaces}

As it follows from Comessatti's classification of real rational surfaces, 
every maximal real rational surface is either the real projective plane $\PP^2$, or a real ruled surface 
$\PP_{\PP^1}(\cal E)$ as precised in Corollary \ref{rational}, or a real del Pezzo surface $X$ with $K_X^2=1$ 
and $X(\RR)$ homeomorphic to $\PP^2(\RR)\sqcup 4S^2$, or any of them blown-up at some number of 
real points. Since $\beta_1(X)$ and the number of real connected components are not changing
under blow-ups, to each of these surfaces we may apply Theorem \ref{main} and find the following 
complement to Corollary \ref{rational}.

\begin{prop}\label{delPezzo} 
The Hilbert square $X^{[2]}$ is not maximal if 
$X$ is a real del Pezzo surface with $K_X^2=1$ and $X(\RR)$ homeomorphic to $\PP^2(\RR)\sqcup 4S^2$, 
or its blowup at some number of real points. For all other maximal real rational surfaces $X$, their Hilbert 
square is maximal.
\end{prop}

\subsubsection{Product of curves}

According to Theorem  \ref{b1>0}, if $X$ is a maximal real surface such that  $b_0(X(\RR))>1+\beta_1,$ 
its Hilbert square is not maximal. Examples of such surfaces with  $\beta_1>0$ exist in abundance.  
The simplest ones  are the products $C_1\times C_2,$ where  $C_1$ and $C_2$ are maximal 
curves of genus $g_1$ and  $g_2,$ respectively, where  $g_1,g_2\geq 2$ and $g_1+g_2>4.$

\subsubsection{Enriques surfaces}

For surfaces $X$ with $\Tors_2 H^2(X;\ZZ)\neq 0,$ one can still compute the $\FF_2$-Betti numbers of the 
Hilbert square $X^{[2]}$ for a given example,  albeit a general formula is lacking \cite[Example 2.5]{square}.
We address here the maximality of the Hilbert square for real Enriques surfaces, and we find:

\begin{thm}
No real Enriques surfaces has maximal Hilbert square.
\end{thm}

\begin{proof} The proof follows the same lines as that of Theorem \ref{b1>0}. There are only $2$ essential 
changes. First, the formula (\ref{totaro-notorsion}) is to be replaced by 
$$
\beta_*(X^{[2]})=\frac12 \beta_*(\beta_*+1)+\beta_*+ 2\beta_1
$$
which follows, for example, from the computation of $\FF_2$-Betti numbers performed by B. Totaro \cite{square}.
Second, at the final step, Lemma \ref{rank-mu-notorsion} is to be replaced by a much simpler observation: 
$\rank\mu\ge \rank \oplus _1^r\mu_i =r$.
\end{proof}

\begin{rmk}
{\rm 
As mentioned in Introduction, the real locus of a  maximal real Enriques surface is always disconnected. 
This can be proved directly, i.e., without appealing to a full classification \cite{enriques},  in the following manner. 
Namely, let us assume that the real locus $X(\RR)$ of some maximal real Enriques surface $X$ is connected. 
Then, $X(\RR)$ is homeomorphic to a real projective plane blown-up at $16-3=13$ real points. Its universal 
covering is a double covering $Y\to X,$ where $Y$ is a K3 surface. $Y$ can be equipped with two possible lifts 
of the real structure. We choose the one for which $Y(\RR)\neq \emptyset.$ Then, $Y(\RR)$ is the orientation 
double cover of $X(\RR),$ and so $\chi(Y(\RR))=2\chi(X(\RR))= -24.$ Since $Y(\RR)$ is orientable, 
it follows that $Y(\RR)$ is homeomorphic to a sphere with 13 handles, which contradicts the Smith inequality 
$\beta_*(Y(\RR))\le \beta_*(Y)=24$.
}
\end{rmk}


\subsection{Application to cubic 4-folds}


The deformation classification of real nonsingular cubic 4-folds, established by S.~Finashin and 
V.~Kharlamov \cite{4fold}, associates with all but one class (called by them {\it irregular}) a deformation class 
of a real nonsingular K3 surface. On the other hand, due to V.~Krasnov \cite{krasnov}, for each regular class 
of cubic 4-folds (in Finashin-Kharlamov's sense) except the one which corresponds to a K3 with 10 spheres 
as the real locus, the Fano variety is equivariantly diffeomorphic to the Hilbert square of the corresponding K3. 
In a separate paper \cite{krasnov+}, Krasnov proved that for the cubics in the irregular class, the real locus 
of the Fano variety is a disjoint union of six disjoint copies of $S^2\times S^2$ and one more component which 
is homeomorphic to $X^{[2]}_{\text{main}}(\RR),$ where $X$ is a K3-surface with $X(\RR)=3S^2$. Therefore, 
combining these results with Corollary \ref{main-corollary} and the computations from the proof of Theorem 
\ref{main}, we get the following statement.

\begin{thm} 
\label{cubic4folds}
If a real nonsingular cubic 4-fold belongs to a regular class and the real locus of the associated real 
K3 is not the union of 10 spheres, as well as if the cubic belongs to the irregular class, then the Fano
variety of the cubic is not maximal.
\end{thm}

\begin{rmk} 
{\rm 
Notice that maximal real nonsingular cubic 4-folds do exist. They form 3 real deformation classes, see \cite{4fold}, 
distinguished by the topology of their real loci: 
$\PP^4(\RR)\# 10(S^2\times S^2)\# (S^1 \times S^3)$, or $\PP^4(\RR)\# 6(S^2\times S^2)\# 5(S^1 \times S^3)$, 
or $\PP^4(\RR)\# 2(S^2\times S^2)\#9 (S^1 \times S^3)$, see \cite{top4fold}. 
As a consequence of Theorem \ref{cubic4folds}, the Fano variety of a maximal real nonsingular
cubic $4$-fold is not maximal. This result contrasts with Krasnov's theorem \cite{krasnov++} stating that the 
Fano surface of a real nonsingular cubic $3$-fold is maximal if and only if the cubic is maximal.
}
\end{rmk}


\subsection{On third Hilbert power}


As was noticed by L.~Fu \cite{fu}, maximality of $X^{[2]}$ implies that of $X^{[3]}$. Combining this 
with Theorem \ref{main} and Propositions \ref{elliptic}, \ref{ruled}, each time when the corresponding 
statement insures maximality of $X^{[2]}$ we deduce maximality of $X^{[3]}$. But it leaves open the 
non-maximality results: for example, is $X^{[3]}$ non-maximal for K3-surfaces, and, more generally, 
for surfaces with $H_1(X)=0$ and disconnected real part?


\bibliographystyle{alpha}

\end{document}